\documentclass[12pt]{amsart}
\usepackage{amsfonts}
\usepackage{amssymb,latexsym}
\usepackage{amsmath}
\usepackage{amsthm}
\usepackage{amssymb}
\usepackage{enumerate}
\newtheorem{theorem}{Theorem}[section]
\newtheorem{lemma}[theorem]{Lemma}
\newtheorem{proposition}[theorem]{Proposition}

\theoremstyle{definition}

\newtheorem{example}[theorem]{Example}

\theoremstyle{remark}
\newtheorem{remark}[theorem]{Remark}
\numberwithin{equation}{section}

\begin{document}
\setlength{\baselineskip}{1.2\baselineskip}
\title [$C^2$ interior a priori estimate ]
{The interior $C^2$ estimate for Monge-Amp\`{e}re equation in dimension $n=2$}
\author{Chuanqiang Chen}
\address{Department of Applied Mathematics,\\
         Zhejiang University of Technology\\
        Hangzhou, 310023, Zhejiang Province, CHINA}
\email{cqchen@mail.ustc.edu.cn}
\author{Fei Han}
\address{School of Mathematics Sciences\\
         Xinjiang Normal University\\
         Urumqi, 830054, Xinjiang Uygur Autonomous Region, CHINA}
\email{klhanqingfei@126.com}
\author{Qianzhong Ou}
\address{Department of Mathematics\\
          Hezhou University\\
         Hezhou, 542800, Guangxi Province, CHINA}
\email{ouqzh@163.com}
\thanks{2010 Mathematics Subject Classification: 35J96, 35B45, 35B65}
\thanks{Keywords: interior $C^2$ estimate, Monge-Amp\`{e}re equation, $\sigma_2$ Hessian equation, optimal concavity}
\thanks{Research of the first author is supported by the National Natural Science Foundation of China (NO. 11301497 and NO. 11471188). Research of the second author is supported by the National Natural Science Foundation of China (NO. 11161048). Research of the third author is supported by NSFC No. 11061013 and by Guangxi Science Foundation (2014GXNSFAA118028) and Guangxi Education Department Key Laboratory of Symbolic Computation and Engineering Data Processing.}
\maketitle

\begin{abstract}
In this paper, we introduce a new auxiliary function, and establish the interior $C^2$ estimate for Monge-Amp\`{e}re equation in dimension $n =2$, which was firstly proved by Heinz \cite{H59} by a geometric method.
\end{abstract}

\section{Introduction}
In this paper, we consider the convex solution of Monge-Amp\`{e}re equation
\begin{equation}\label{1.1}
\det D^2u=f(x), \quad \text{ in }  B_R(0) \subset \mathbb{R}^2.
\end{equation}
When the solution $u$ is convex, the equation \eqref{1.1} is elliptic. It is well-known that the interior $C^2$ estimate is an important problem for elliptic equations. For Monge-Amp\`{e}re equation in dimension $n=2$, the corresponding interior $C^2$ estimate was established by Heinz \cite{H59}, and for higher dimension $n \geq 3$, Pogorelov \cite{P78} constructed his famous counterexample, namely irregular solutions to Monge-Amp\`{e}re equations.

Later, Urbas \cite{U90} generalized the counterexample for $\sigma_k$ Hessian equations with $k \geq 3$. So the interior $C^2$ estimate of $\sigma_2$ Hessian equation \begin{equation}\label{1.2}
\sigma_2(D^2 u)=f, \quad \text{ in }  B_R(0) \subset \mathbb{R}^n,
\end{equation}
is an interesting problem, where
$\sigma_2(D^2u)=\sigma_2(\lambda(D^2 u) )= \sum\limits_{1 \le i_1 < i_2 \leq n}\lambda_{i_1}\lambda_{i_2}$, $\lambda(D^2 u)=(\lambda_1, \cdots, \lambda_n)$ are the eigenvalues of $D^2 u$, and $f >0$. For $n=2$, \eqref{1.2} is the Monge-Amp\`{e}re equation \eqref{1.1}. For $n =3$ and $f \equiv 1$, \eqref{1.2} can be viewed as a special Lagrangian equation, and Warren-Yuan obtained the corresponding interior $C^2$ estimate in the celebrated paper \cite{WY09}. Moreover, the problem is still open for general $f$ with $n \geq 4$ and nonconstant $f$ with $n =3$.

Moreover, Pogorelov type estimates for the Monge-Amp\`{e}re equations and $\sigma_k$ Hessian equation ($k \geq 2$) were derived by Pogorelov \cite{P78} and Chou-Wang \cite{CW01} respectively, and see \cite{GRW14} and \cite{LRW15} for some generalizations.

In this paper, we introduce a new auxiliary function, and establish the interior $C^2$ estimate as follows
\begin{theorem}\label{th1.1}
Suppose $u \in C^4(B_R(0))$ be a convex solution of Monge-Amp\`{e}re equation \eqref{1.1} in dimension $n=2$, where $0 < m \leq f \leq M$ in $B_R(0)$. Then
\begin{align}\label{1.3}
|D^2 u(0)| \leq C_1 e^{C_2 \frac{\sup |Du|^2}{R^2}},
\end{align}
where $C_1$ is a positive constant depending only on $m$, $M$, $R \sup |\nabla f|$, $R^2 \sup |\nabla^2 f|$, and $C_2$ is a positive constant depending only on $m$ and $M$.
\end{theorem}

\begin{remark}
By Trudinger's gradient estimates (see \cite{T97}), we can bound $|D^2 u(0)|$ in terms of $u$. In fact, we can get from the convexity of $u$
\begin{align*}
\mathop {\sup }\limits_{B_{\frac{R}{2}} (0)} |Du| \leq \frac{\mathop {osc }\limits_{B_{R}(0) } u}{\frac{R}{2}} \leq \frac{4\mathop {\sup }\limits_{B_{R} (0)}|u| }{R},
\end{align*}
and
\begin{align}\label{1.4}
|D^2 u(0)| \leq C_1 e^{C_2 \frac{\mathop {\sup }\limits_{B_{\frac{R}{2}}(0) } |Du|^2}{(\frac{R}{2})^2}}\leq C_1 e^{16 C_2 \frac{ \mathop {\sup }\limits_{B_{R}(0) }|u|^2} {R^4}}.
\end{align}
\end{remark}

\begin{remark}
The result was firstly proved by Heinz \cite{H59}. In fact, Heinz's proof depends on the strict convexity of solutions and the geometry of convex hypersurface in dimension two. Our proof, which is based on a suitable choice of auxiliary functions, is elementary and avoids geometric
computations on the graph of solutions.
\end{remark}

\begin{remark}
The interior $C^2$ estimate of $\sigma_2$ Hessian equation \eqref{1.2} in higher dimensions is a longstanding problem. As we all know, it is hard to find a corresponding geometry in higher dimensions, so we can not generalize Heinz's proof or Warren-Yuan's proof to higher dimensions. But the method in this paper and the optimal concavity in \cite{C13} is helpful for this problem.
\end{remark}

The rest of the paper is organized as follows. In Section 2, we give the calculations of the derivatives of eigenvalues and eigenvectors with respect to the matrix. In Section 3, we introduce a new auxiliary function, and prove Theorem \ref{th1.1}.

\section{Derivatives of eigenvalues and eigenvectors}

In this section, we give the calculations of the derivatives of eigenvalues and eigenvectors with respect to the matrix. We think the following result is known for many people, for example see \cite{A07} for a similar result. For completeness, we give the result and a detailed proof.

\begin{proposition}\label{prop2.1}
Let $W= \{W_{ij}\}$ is an $n \times n$ symmetric matrix and $ \lambda(W)= (\lambda_1,\lambda _2, \cdots ,\lambda _{n})$ are the eigenvalues
of the symmetric matrix $W$, and the corresponding continuous eigenvector field is $\tau^i=({\tau^i}_{,1}, \cdots, {\tau^i}_{,n}) \in \mathbb{S}^{n-1}$. Suppose that $W= \{W_{ij}\}$ is diagonal, $\lambda_i= W_{ii}$ and the corresponding eigenvector $\tau^i=(0, \cdots ,0,\mathop 1\limits_{i^{th} } ,0, \cdots ,0) \in \mathbb{S}^{n-1}$ at the diagonal matrix $W$. If $\lambda_k$ is distinct with other eigenvalues, then we have at the diagonal matrix $W$
\begin{align}
\label{2.1}&\frac{{\partial {\tau^k} _{,k} }}{{\partial W_{pq} }} = 0, ~\forall ~p, q; \quad \frac{{\partial {\tau^k} _{,i} }}{{\partial W_{ik} }} = \frac{1}{\lambda_k - \lambda_i}, \quad i \ne k; \quad \frac{{\partial {\tau^k}_{,i} }}{{\partial W_{pq} }} = 0, \text{  otherwise}. &\\
\label{2.2}& \frac{{\partial ^2 {\tau^k}_{,k} }}
{{\partial W_{pk} \partial W_{pk} }} =  - \frac{1} {{(\lambda _k  - \lambda _p )^2 }}, \quad p \ne k;\\
\label{2.3}&\frac{{\partial ^2 {\tau^k} _{,i} }} {{\partial W_{ik} \partial W_{ii} }} = \frac{1}
{{(\lambda _k  - \lambda _i )^2 }}, ~i \ne k; \quad \frac{{\partial ^2 {\tau^k}_{,i} }}
{{\partial W_{ik} \partial W_{kk} }} =  - \frac{1}{{(\lambda _k  - \lambda _i )^2 }}, \quad i \ne k;\\
\label{2.4}& \frac{{\partial ^2 {\tau^k}_{,i} }}{{\partial W_{iq} \partial W_{qk} }} = \frac{1}
{{\lambda _k  - \lambda _i }}\frac{1}{{\lambda _k  - \lambda _q }}, \quad i \ne k,i \ne q,q \ne k;\\
\label{2.5}& \frac{{\partial ^2 {\tau^k}_{,i} }} {{\partial W_{pq} \partial W_{rs} }} = 0, \text{ otherwise}.
\end{align}
\end{proposition}

\begin{proof}
From the definition of eigenvalue and eigenvector of matrix $W$, we have
\begin{align*}
(W - \lambda _k I){\tau^k} \equiv 0,
\end{align*}
where ${\tau^k} $ is the eigenvector of $W$ corresponding to the eigenvalue $\lambda_k$. That is, for $i = 1, \cdots, n$, it holds
\begin{align}\label{2.6}
  [W_{ii}  - \lambda _k ]{\tau^k} _{,i}  + \sum\limits_{j \ne i} {W_{ij} } {\tau^k} _{,j}  = 0.
 \end{align}

When $W= \{W_{ij}\}$ is diagonal and $\lambda_k$ is distinct with other eigenvalues, $\lambda_k$ and $\tau^k$ are $C^2$ at the matrix $W$. In fact,
\begin{align}\label{2.7}
{\tau^k}_{,k}=1, \quad {\tau^k}_{,i} =0, \quad i \ne k, \quad \text{ at } W.
\end{align}

Taking the first derivative of \eqref{2.6}, we have
\begin{align*}
  [\frac{{\partial W_{ii} }}
{{\partial W_{pq} }} - \frac{{\partial \lambda _k }}
{{\partial W_{pq} }}]{\tau^k} _{,i}  + [W_{ii}  - \lambda _k ]\frac{{\partial {\tau^k}_{,i} }}
{{\partial W_{pq} }} + \sum\limits_{j \ne i} {[\frac{{\partial W_{ij} }}
{{\partial W_{pq} }}{\tau^k}_{,j}  + W_{ij} \frac{{\partial {\tau^k} _{,j} }}
{{\partial W_{pq} }}]}  = 0.
 \end{align*}
Hence for $i=k$, we get from \eqref{2.7}
\begin{align}\label{2.8}
  \frac{{\partial \lambda _k }}
{{\partial W_{pq} }} = \frac{{\partial W_{kk} }}
{{\partial W_{pq} }} = \left\{ \begin{gathered}
  \begin{array}{*{20}c}
   {1,} & {p = k,q = k;}  \\
 \end{array}  \hfill \\
  \begin{array}{*{20}c}
   {0,} & {otherwise.}  \\
 \end{array}  \hfill \\
\end{gathered}  \right.
 \end{align}
And for $i \ne k$,
\begin{align*}
  [W_{ii}  - \lambda _k ]\frac{{\partial {\tau^k}_{,i} }}
{{\partial W_{pq} }} + \sum\limits_{j \ne i} {\frac{{\partial W_{ij} }}
{{\partial W_{pq} }}{\tau^k}_{,j} }  = 0,
 \end{align*}
then
\begin{align}\label{2.9}
  \frac{{\partial {\tau^k}_{,i} }}
{{\partial W_{pq} }} = \frac{1}
{{\lambda _k  - \lambda _i }}\frac{{\partial W_{ik} }}
{{\partial W_{pq} }} = \left\{ \begin{gathered}
  \begin{array}{*{20}c}
   {\frac{1}
{{\lambda _k  - \lambda _i }},} & {p = i,q = k;}  \\
 \end{array}  \hfill \\
  \begin{array}{*{20}c}
   {0,} & {\qquad otherwise.}  \\
\end{array}  \hfill \\
\end{gathered}  \right.
 \end{align}

Since $\tau^k \in \mathbb{S}^{n-1}$, we have
\begin{align} \label{2.10}
  1 = |\tau ^k |^2  = ({\tau^k}_{,1} )^2  +  \cdots  + ({\tau^k}_{,k} )^2  +  \cdots  + ({\tau^k}_{,n} )^2.
\end{align}
Taking the first derivative of \eqref{2.10}, and using \eqref{2.7}, it holds
\begin{align}\label{2.11}
  \frac{{\partial {\tau^k}_{,k} }}{{\partial W_{pq} }} = 0, \quad \forall (p,q).
\end{align}

For $i=k$, taking the second derivative of \eqref{2.6}, and using \eqref{2.7}, it holds
\begin{align*}
  [\frac{{\partial ^2 W_{kk} }}
{{\partial W_{pq} \partial W_{rs} }} - \frac{{\partial ^2 \lambda _k }}
{{\partial W_{pq} \partial W_{rs} }}]{\tau^k}_{,k}  + \sum\limits_{j \ne k} {[\frac{{\partial W_{kj} }}
{{\partial W_{pq} }}\frac{{\partial {\tau^k}_{,j} }}
{{\partial W_{rs} }} + \frac{{\partial W_{kj} }}
{{\partial W_{rs} }}\frac{{\partial {\tau^k} _{,j} }}
{{\partial W_{pq} }}]}  = 0,
 \end{align*}
hence
\begin{align} \label{2.12}
  \frac{{\partial ^2 \lambda _k }}
{{\partial W_{pq} \partial W_{rs} }} =& \sum\limits_{j \ne k} {[\frac{{\partial W_{kj} }}
{{\partial W_{pq} }}\frac{{\partial {\tau^k} _{,j} }}
{{\partial W_{rs} }} + \frac{{\partial W_{kj} }}
{{\partial W_{rs} }}\frac{{\partial {\tau^k} _{,j} }}
{{\partial W_{pq} }}]}  \notag \\
 =& \left\{ \begin{gathered}
  \begin{array}{*{20}c}
   {\frac{1}
{{\lambda _k  - \lambda _q }},} & {p = k,q \ne k,r = q,s = k;}  \\
 \end{array}  \hfill \\
  \begin{array}{*{20}c}
   {\frac{1}
{{\lambda _k  - \lambda _s }},} & {r = k,s \ne k,p = s,q = k;}  \\
 \end{array}  \hfill \\
  \begin{array}{*{20}c}
   {0,} & {\qquad otherwise.}  \\
 \end{array}  \hfill \\
\end{gathered}  \right.
 \end{align}
For $i \ne k$, it holds
\begin{align*}
 & [\frac{{\partial W_{ii} }}
{{\partial W_{pq} }} - \frac{{\partial \lambda _k }}
{{\partial W_{pq} }}]\frac{{\partial {\tau^k}_{,i} }}
{{\partial W_{rs} }} + [\frac{{\partial W_{ii} }}
{{\partial W_{rs} }} - \frac{{\partial \lambda _k }}
{{\partial W_{rs} }}]\frac{{\partial {\tau^k}_{,i} }}
{{\partial W_{pq} }} + [W_{ii}  - \lambda _k ]\frac{{\partial ^2 {\tau^k}_{,i} }}
{{\partial W_{pq} \partial W_{rs} }}  \\
&+ \sum\limits_{j \ne i} {[\frac{{\partial W_{ij} }}
{{\partial W_{pq} }}\frac{{\partial {\tau^k}_{,j} }}
{{\partial W_{rs} }} + \frac{{\partial W_{ij} }}
{{\partial W_{rs} }}\frac{{\partial {\tau^k}_{,j} }}
{{\partial W_{pq} }}]}  = 0,
 \end{align*}
then
\begin{align}
\label{2.13}&\frac{{\partial ^2 {\tau^k}_{,i} }}{{\partial W_{ik} \partial W_{ii} }} = \frac{1}
{{\lambda _k  - \lambda _i }}\frac{{\partial {\tau^k}_{,i} }}{{\partial W_{ik} }} = \frac{1}
{{\lambda _k  - \lambda _i }}\frac{1}{{\lambda _k  - \lambda _i }}, \quad i \ne k;   \\
\label{2.14}&\frac{{\partial ^2 {\tau^k}_{,i} }}{{\partial W_{iq} \partial W_{qk} }} = \frac{1}
{{\lambda _k  - \lambda _i }}\frac{{\partial W_{iq} }}
{{\partial W_{iq} }}\frac{{\partial {\tau^k}_{,q} }}
{{\partial W_{qk} }} = \frac{1}{{\lambda _k  - \lambda _i }}\frac{1}
{{\lambda _k  - \lambda _q }}, \quad i \ne k,i \ne q,q \ne k; \\
\label{2.15}&\frac{{\partial ^2 {\tau ^k} _{,i} }} {{\partial W_{ik} \partial W_{kk} }} = \frac{1}
{{\lambda _k  - \lambda _i }}[ - \frac{{\partial \lambda _k }}
{{\partial W_{kk} }}\frac{{\partial {\tau ^k} _{,i} }}{{\partial W_{ik} }}] =  - \frac{1}
{{\lambda _k  - \lambda _i }}\frac{1}{{\lambda _k  - \lambda _i }}, \quad i \ne k; \\
\label{2.16}&\frac{{\partial ^2 {\tau^k}_{,i} }}{{\partial W_{pq} \partial W_{rs} }} = 0, \quad \text{ otherwise. }
\end{align}

From \eqref{2.10}, we have
\begin{align*}
2\tau ^k _{,k} \frac{{\partial ^2 {\tau^k}_{,k} }}
{{\partial W_{pq} \partial W_{rs} }} + 2\sum\limits_{i \ne k} {\frac{{\partial {\tau^k}_{,i} }}
{{\partial W_{pq} }}\frac{{\partial {\tau^k}_{,i} }}
{{\partial W_{rs} }}}  = 0,
 \end{align*}
then
\begin{align*}
\frac{{\partial ^2 {\tau^k}_{,k} }}
{{\partial W_{pq} \partial W_{rs} }} =  - \sum\limits_{i \ne k} {\frac{{\partial {\tau^k}_{,i} }}
{{\partial W_{pq} }}\frac{{\partial {\tau^k}_{,i} }}
{{\partial W_{rs} }}}  = \left\{ \begin{gathered}
  \begin{array}{*{20}c}
   { - \frac{1}
{{\lambda _k  - \lambda _p }}\frac{1}
{{\lambda _k  - \lambda _p }},} & {p \ne k,q = k,r = p,s = q;}  \\
 \end{array}  \hfill \\
  \begin{array}{*{20}c}
   {0,} & {\qquad \qquad \quad otherwise.}  \\
 \end{array}  \hfill \\
\end{gathered}  \right.
\end{align*}
The proof of Proposition \ref{prop2.1} is finished.
\end{proof}

\begin{example}
When $n =2$, the matrix $
\left( {\begin{array}{*{20}c}
   {u_{11} } & {u_{12} }  \\
   {u_{21} } & {u_{22} }  \\

 \end{array} } \right)
$
has two eigenvalues
\[
\begin{gathered}
  \lambda _1  = \frac{{(u_{11}  + u_{22} ) + \sqrt {(u_{11}  - u_{22} )^2  + 4u_{12} u_{21} } }}
{2}, \hfill \\
  \lambda _2  = \frac{{(u_{11}  + u_{22} ) - \sqrt {(u_{11}  - u_{22} )^2  + 4u_{12} u_{21} } }}
{2}, \hfill \\
\end{gathered}
\]
with $\lambda_1 \geq \lambda_2$. If $\lambda_1 > \lambda_2$,
\[
\left[ {\left( {\begin{array}{*{20}c}
   {u_{11} } & {u_{12} }  \\
   {u_{21} } & {u_{22} }  \\

 \end{array} } \right) - \lambda _1 \left( {\begin{array}{*{20}c}
   1 & 0  \\
   0 & 1  \\

 \end{array} } \right)} \right]\left( {\begin{array}{*{20}c}
   {\xi _1 }  \\
   {\xi _2 }  \\

 \end{array} } \right) = 0,
\]
we can get
\[
  \xi _1  = \frac{{(u_{22}  - u_{11} ) - \sqrt {(u_{11}  - u_{22} )^2  + 4u_{12} u_{21} } }}
{2}; \qquad    \xi _2  =  - u_{21}.
\]
Then the eigenvector $\tau$ corresponding to $\lambda_1$ is
\[
\tau  =  - \frac{{(\xi _1 ,\xi _2 )}} {{\sqrt {{\xi _1 }^2  + {\xi _2} ^2 } }}.
\]
We can verify Proposition \ref{prop2.1}.
\end{example}

\section{Proof of Theorem \ref{th1.1}}

Now we start to prove Theorem \ref{th1.1}.

Let $\tau (x) = \tau(D^2 u(x)) =(\tau_1, \tau_2) \in \mathbb{S}^{1}$ be the continuous eigenvector field of $D^2 u(x)$ corresponding to the largest eigenvalue. Denote
\begin{align}\label{3.2}
\Sigma =: \{x \in B_R(0): r^2 -|x|^2 + \langle x, \tau(x)\rangle^2 >0, r^2 - \langle x, \tau(x)\rangle^2 >0\},
\end{align}
where $r = \frac{1}{\sqrt 2} R$. It is easy to know, $\Sigma$ is an open set and $ B_{r}(0) \subset \Sigma \subset  B_R(0)$.
We introduce a new auxiliary function in $\Sigma$ as follows
\begin{equation} \label{3.1}
\phi(x)  = \eta(x)^\beta g(\frac{1}{2}|Du|^2 )u_{\tau \tau}
\end{equation}
where $\eta (x) = (r^2 -|x|^2 + \langle x, \tau(x)\rangle^2 )(r^2 - \langle x, \tau(x)\rangle^2)$ with $\beta = 4$ and $g(t) = e^{\frac{c_0}{r^2}t}$ with $c_0 = \frac{32}{m}$. In fact, $\langle x, \tau(x)\rangle$ is invariant under rotations of the coordinates, so is $\eta(x)$.

From the definition of $\Sigma$, we know $\eta(x) >0$ in $\Sigma$, and $\eta =0 $ on $\partial \Sigma$. Assume the maximum of $\phi(x)$ in $\Sigma$ is attains at $x_0 \in \Sigma$. By rotating the coordinates, we can assume $D^2u (x_0)$ is diagonal. In the following, we denote $\lambda_i = u_{ii}(x_0)$, $\lambda = (\lambda_1, \lambda_2)$. Without loss of generality, we can assume $\lambda_1 \geq \lambda_2 $, and $\tau(x_0) = (1, 0)$.

If $ \eta \lambda_1 \leq 10^3 (1+ M + r \sup |\nabla f| + \frac{M}{m}\frac{\sup |Du|}{r})r^4$. Then we can easily get
\begin{align*}
u_{\tau(0) \tau(0)}(0) \leq& \frac{1}{r^{4\beta}} \phi (0) \leq \frac{1}{r^{4\beta}} \phi (x_0) \\
 \leq& 10^3 (1+ M+ r \sup |\nabla f| + \frac{M}{m}\frac{\sup |Du|}{r}) e^{c_0\frac{\sup |Du|^2}{r^2}} \\
 \leq& 10^3 (1+ M + r \sup |\nabla f|) e^{(c_0+\frac{2M}{m})\frac{\sup |Du|^2}{r^2}}.
\end{align*}
Hence we get
\begin{align}
|u_{\xi \xi} (0) |\leq  u_{\tau(0) \tau(0)}(0) \leq  10^3 (1+ M + r \sup |\nabla f|) e^{(c_0+\frac{2M}{m})\frac{\sup |Du|^2}{r^2}}, \quad \forall ~\xi \in \mathbb{S}^{1}. \notag
\end{align}
Then we prove Theorem \ref{th1.1} under the condition $\eta \lambda_1 \leq 10^3 (1+ M + r \sup |\nabla f| + \frac{M}{m}\frac{\sup |Du|}{r})r^4$.

Now, we assume $\eta \lambda_1 \geq 10^3 (1+ M + r \sup |\nabla f| + \frac{M}{m}\frac{\sup |Du|}{r})r^4$. Then we have
\begin{align}\label{3.3}
\lambda_1 = \frac{\eta \lambda_1}{\eta} \geq  10^3 (1+ M + r \sup |\nabla f| + \frac{M}{m}\frac{\sup |Du|}{r}).
\end{align}
From the equation \eqref{1.1}, we have
\begin{align*}
\lambda_2 = \frac{f}{\lambda_1} \leq  \frac{M}{\lambda_1} < \lambda_1.
\end{align*}
Hence $\lambda_1 $ is distinct with the other eigenvalue, and $\tau (x)$ is $C^2$ at $x_0$. Moreover, the test function
\begin{equation} \label{3.4}
\varphi  = \beta \log \eta  + \log g(\frac{1}{2}|Du|^2 ) + \log u_{11}
\end{equation}
attains the local maximum at $x_0$. In the following, all the calculations are at $x_0$.

Then, we can get
\begin{equation}
0= \varphi _i  = \beta \frac{{\eta _i }}{\eta } + \frac{{g'}}{g}\sum\limits_k
{u_k u_{ki} }  + \frac{{u_{11i} }}{{u_{11} }} , \notag
\end{equation}
so we have
\begin{equation} \label{3.5}
\frac{{u_{11i} }}{{u_{11} }} =  - \beta  \frac{{\eta _i }}{\eta } -
\frac{{g'}}{g}u_i u_{ii}, \quad i =1, 2.
\end{equation}
At $x_0$, we also have
\begin{align*}
0 \geq \varphi _{ii}  =& \beta [ \frac{{\eta _{ii} }}{\eta } - \frac{{\eta _i ^2
}}{{\eta ^2 }}] + \frac{{g''g - g'^2 }}{{g^2 }}\sum\limits_k {u_k
u_{ki} } \sum\limits_l {u_l u_{li} } \notag \\
&+ \frac{{g'}}{g}\sum\limits_k {(u_{ki} u_{ki}  + u_k u_{kii} )}  +
\frac{{u_{11ii} }}{{u_{11} }} - \frac{{u_{11i} ^2 }}{{u_{11} ^2 }}
\notag \\
=& \beta  [\frac{{\eta _{ii} }}{\eta } - \frac{{\eta _i ^2 }}{{\eta ^2 }} ]+ \frac{{g'}}{g}[u_{ii} ^2  + \sum\limits_k {u_k u_{kii} } ] +
\frac{{u_{11ii} }}{{u_{11} }} - \frac{{u_{11i} ^2 }}{{u_{11} ^2 }},
\end{align*}
since $g''g - g'^2 =0$. Let
\begin{align*}
&F^{11} = \frac{\partial \det D^2 u}{\partial u_{11}} = \lambda_2, \quad F^{22} = \frac{\partial \det D^2 u}{\partial u_{22}} = \lambda_1, \\
&F^{12} = \frac{\partial \det D^2 u}{\partial u_{12}} = 0, \quad F^{21} = \frac{\partial \det D^2 u}{\partial u_{21}} = 0.
\end{align*}
Then from the equation \eqref{1.1} we can get
\begin{align}\label{3.6}
\lambda_2 = \frac{f}{\lambda_1}.
\end{align}
Differentiating \eqref{1.1} once, we can get
\begin{align}
F^{11} u_{11i} +F^{22} u_{22i} = f_i, \notag
\end{align}
then
\begin{align}\label{3.7}
 u_{22i} = \frac{1}{F^{22}}[f_i - F^{11} u_{11i} ] = \frac{f_i }{\lambda_1} - \frac{f}{\lambda_1} \frac{u_{11i}}{u_{11}}.
\end{align}
Differentiating \eqref{1.1} twice, we can get
\begin{align}\label{3.8}
F^{11} u_{1111} +F^{22} u_{2211}  =& f_{11} - 2\frac{\partial^2 \det D^2 u }{\partial u_{11} \partial u_{22}}u_{111}u_{221} - 2 \frac{\partial^2 \det D^2 u }{\partial u_{12} \partial u_{21}}u_{112}^2  \notag \\
=& f_{11} - 2u_{111}u_{221} + 2 u_{112}^2  \notag \\
=&f_{11} + 2 u_{112}^2 - 2u_{111}[\frac{f_1 }{\lambda_1} - \frac{f}{\lambda_1} \frac{u_{111}}{u_{11}}] \notag  \\
=&f_{11} + 2 u_{112}^2 - 2 f_1 \frac{ u_{111}}{u_{11}} + 2 f (\frac{u_{111}}{u_{11}})^2,
\end{align}
and
\begin{align}\label{3.9}
F^{11} u_{1112} +F^{22} u_{2212}=& f_{12} - \frac{\partial^2 \det D^2 u }{\partial u_{11} \partial u_{22}}u_{111}u_{222}- \frac{\partial^2 \det D^2 u }{\partial u_{22} \partial u_{11}}u_{221}u_{112} \notag\\
&- 2\frac{\partial^2 \det D^2 u }{\partial u_{12} \partial u_{21}} u_{121}u_{212}  \notag \\
=& f_{12} - u_{111}u_{222}- u_{112}u_{221} + 2 u_{112}u_{221}  \notag \\
=& f_{12} - u_{111} [\frac{f_2 }{\lambda_1} - \frac{f}{\lambda_1} \frac{u_{112}}{u_{11}}] +  u_{112}[\frac{f_1 }{\lambda_1} - \frac{f}{\lambda_1} \frac{u_{111}}{u_{11}}] \notag  \\
=&f_{12} + f_1 \frac{ u_{112}}{u_{11}} - f_2  \frac{u_{111}}{u_{11}}.
\end{align}

Hence
\begin{eqnarray} \label{3.10}
0 &\ge& \sum\limits_{i = 1}^2 {F^{ii} \varphi _{ii} }  \notag\\
&=& \beta \sum\limits_{i } {F^{ii} [\frac{{\eta _{ii} }}{\eta } -
\frac{{\eta _i ^2 }}{{\eta ^2 }}]}+ \frac{{g'}}{g}\sum\limits_{i } {F^{ii} u_{ii} ^2 } + \frac{{g'}}{g}\sum\limits_k {u_k f_{k} }  \notag \\
&& +\frac{1}{{u_{11} }}\sum\limits_{i } {F^{ii} u_{11ii} }  -
\sum\limits_{i } {F^{ii} [\frac{{u_{11i} }}{{u_{11} }}]^2 } \notag \\
&=& \beta \lambda_2 [\frac{{\eta _{11}
}}{\eta } - \frac{{\eta _1 ^2 }}{{\eta ^2 }}] +\beta \lambda_1 [\frac{{\eta _{22}
}}{\eta } - \frac{{\eta _2 ^2 }}{{\eta ^2 }}] + \frac{{g'}}{g}[\lambda_1 + \lambda_2]f + \frac{{g'}}{g} [u_1 f_{1} + u_2 f_{2} ] \notag \\
&& + \frac{1}{{u_{11} }}[f_{11} + 2 u_{112}^2 - 2 f_1 \frac{ u_{111}}{u_{11}} + 2 f (\frac{u_{111}}{u_{11}})^2]  - \lambda_2 [\frac{{u_{111} }}{{u_{11} }}]^2  - \lambda_1 [\frac{{u_{112} }}{{u_{11} }}]^2 \notag \\
&\geq&  \beta [\frac{f}{\lambda_1}\frac{{\eta _{11}}}{\eta } + \lambda_1 \frac{{\eta _{22}
}}{\eta } ] - \beta \frac{f}{\lambda_1}\frac{{\eta _1 ^2 }}{{\eta ^2 }} - \beta \lambda_1 \frac{{\eta _2 ^2 }}{{\eta ^2 }} \notag \\
&&+ \frac{f}{ \lambda_1} [\frac{{u_{111} }}{{u_{11} }}]^2  -2 \frac{f_1}{{u_{11} }} \frac{ u_{111}}{u_{11}}+ \frac{ \lambda_1 }{2}[\frac{{u_{112} }}{{u_{11} }}]^2+ \frac{ \lambda_1 }{2}[\beta\frac{{\eta_{2} }}{{\eta }} + \frac{g'}{g} u_2 u_{22}]^2 \notag  \\
&&+ \frac{{g'}}{g}f \lambda_1 - \frac{{g'}}{g} |\nabla u| |\nabla f| - \frac{| f_{11} |}{{\lambda_1 }} \notag \\
&\geq&  \beta [\frac{f}{\lambda_1}\frac{{\eta _{11}}}{\eta } + \lambda_1 \frac{{\eta _{22}
}}{\eta } ] - \beta \frac{f}{\lambda_1}\frac{{\eta _1 ^2 }}{{\eta ^2 }} - \beta \lambda_1 \frac{{\eta _2 ^2 }}{{\eta ^2 }} \notag \\
&&+ \frac{f}{ 2\lambda_1} [\frac{{u_{111} }}{{u_{11} }}]^2  + \frac{ \lambda_1 }{2}[\frac{{u_{112} }}{{u_{11} }}]^2 + \frac{ \beta^2 }{2}\lambda_1 \frac{{\eta_{2}^2 }}{{\eta^2 }} + \beta f \frac{g'}{g}  \frac{{\eta_{2} }}{{\eta }} u_2\notag  \\
&&+ \frac{{g'}}{g}f \lambda_1 - \frac{{g'}}{g} |\nabla u| |\nabla f| - \frac{| f_{11} |}{{\lambda_1 }}- 2\frac{f_1^2}{f\lambda_1}.
\end{eqnarray}

\begin{lemma}\label{lem3.1}
Under the condition $ \eta \lambda_1 \geq 10^3 (1+ M + r \sup |\nabla f| + \frac{M}{m}\frac{\sup |Du|}{r})r^4$, we have at $x_0$
\begin{align}
\label{3.11}&\beta \frac{f}{\lambda_1}\frac{{\eta _1 ^2 }}{{\eta ^2 }} \leq \frac{8\beta f r^6 }{\eta^2\lambda_1} +\frac{\lambda_1}{4}(\frac{u_{112}} {{u_{11}  }})^2;   \\
\label{3.12}&\beta f \frac{g'}{g}  \frac{{\eta_{2} }}{{\eta }} u_2 \geq - 4\beta f \frac{g'}{g} \frac{ r^4}{\eta} \frac{|u_2|}{r};
\end{align}
and
\begin{align}\label{3.13}
\beta [\frac{f}{\lambda_1}\frac{{\eta _{11}}}{\eta } + \lambda_1 \frac{{\eta _{22}
}}{\eta } ]
\geq& -  \frac{1}{2}\frac{{g'}}{g} f\lambda_1-\beta \lambda_1 [\frac{\eta_2}{\eta}]^2 - \frac{f}{2\lambda_1}(\frac{u_{111}}{u_{11}})^2 -\frac{\lambda_1}{4}(\frac{u_{112}} {{u_{11} }})^2  \notag\\
&- 2\beta f \frac{{r^2}}{\eta \lambda_1} - 4\beta |f_{12}| \frac{{r^4}}{\eta \lambda_1} -32 \beta \frac{f_1 ^2 r^4}{ \eta \lambda_1^3} -8\beta \frac{|f_1 f_2| r^4}{ \eta \lambda_1^3} -24 \beta \frac{ |f_1|r^3}{ \eta \lambda_1} \notag \\
&- [\frac{6 \beta |f_1|^2r^4}{\eta \lambda_1} +\frac{12 \beta f r^2}{\eta \lambda_1} ]  -  [ \frac{{8\beta f r^4 |f_2|^2}}{\eta \lambda_1^3}+ \frac{{(48 \beta)^2 f r^6  }}{\eta^2\lambda_1 }+ \frac{{32 \beta^2 |f_{2}|^2 r^8  }}{\eta^2\lambda_1 f}].
\end{align}
\end{lemma}

\begin{proof}
At $x_0$, $\tau=(\tau_1, \tau_2)= (1, 0)$.
Then from Proposition \ref{prop2.1} we get
\begin{align}\label{3.14}
\left\langle {x,\partial _i \tau } \right\rangle  =&  \sum\limits_{m = 1}^2 {x_m \frac{{\partial \tau _m }}
{{\partial x_i }}}  = \sum\limits_{m = 1}^2 {x_m \frac{{\partial \tau _m }}
{{\partial u_{pq} }}u_{pqi}} = x_2 \frac{{\partial \tau _2 }}
{{\partial u_{pq} }}u_{pqi}   \notag \\
=& x_2 \frac{u_{12i}} {{\lambda_1 - \lambda_2 }}, \quad i = 1, 2.
\end{align}

From the definition of $\eta$, then we have at $x_0$
\begin{align}\label{3.15}
\eta  = [r^2  - |x|^2  + \left\langle {x,\tau } \right\rangle ^2 ][r^2  - \left\langle {x,\tau } \right\rangle ^2 ] =  (r^2  - x_2 ^2)(r^2  - x_1 ^2).
\end{align}
Taking first derivative of $\eta$, we can get
\begin{align*}
\eta _i =& [ - 2x_i  + 2\left\langle {x,\tau } \right\rangle \left\langle {x,\tau } \right\rangle _i ][r^2  - \left\langle {x,\tau } \right\rangle ^2 ]\\
   &+ [r^2  - |x|^2  + \left\langle {x,\tau } \right\rangle ^2 ][ - 2\left\langle {x,\tau } \right\rangle \left\langle {x,\tau } \right\rangle _i ] \\
   =& [ - 2x_i  + 2x_1 ( \delta_{i1}+ \left\langle {x, \partial_i \tau } \right\rangle ) ][r^2  - x_1 ^2 ] +( r^2 -x_2^2) [ - 2x_1  ( \delta_{i1} + \left\langle {x, \partial_i \tau } \right\rangle )  ] \hfill \\
   =& \left\{ \begin{array}{l}
  - 2x_1(r^2  - x_2 ^2 )+  2x_1 \left\langle {x, \partial_1 \tau } \right\rangle (x_2^2 - x_1 ^2), \quad i =1;\\
  - 2x_2(r^2  - x_1 ^2 )+  2x_1 \left\langle {x, \partial_2 \tau } \right\rangle (x_2^2 - x_1 ^2), \quad i =2.\\
 \end{array} \right.
\end{align*}
Hence
\begin{align} \label{3.16}
\beta \frac{f}{\lambda_1}\frac{{\eta _1 ^2 }}{{\eta ^2 }} =&  \beta \frac{f}{\lambda_1} [\frac{- 2x_1(r^2  - x_2^2)}{\eta} + 2x_1 x_2 \frac{x_2^2 - x_1 ^2}{\eta} \frac{u_{112}} {{\lambda_1 - \lambda_2 }}]^2  \notag \\
\leq& \beta \frac{f}{\lambda_1} [\frac{8r^6 }{\eta^2} + \frac{8r^8 }{\eta^2} (\frac{u_{112}} {{u_{11}  }})^2]\notag \\
\leq& \frac{8\beta f r^6 }{\eta^2\lambda_1} +\frac{\lambda_1}{4}(\frac{u_{112}} {{u_{11}  }})^2.
\end{align}
Also we have
\begin{align}
\frac{{\eta _2 }}{{\eta }} =&\frac{- 2x_2(r^2  - x_1^2)}{\eta} + 2x_1 x_2 \frac{x_2^2 - x_1 ^2}{\eta} \frac{u_{221}} {{\lambda_1 - \lambda_2 }} \notag \\
=&\frac{- 2x_2(r^2  - x_1^2)}{\eta} + 2x_1 x_2 \frac{x_2^2 - x_1 ^2}{\eta} \frac{1} {{\lambda_1 - \lambda_2 }} [ \frac{f_1 }{\lambda_1} - \frac{f}{\lambda_1}\frac{u_{111}} {{u_{11}  }}]\notag \\
=&\frac{- 2x_2(r^2  - x_1^2)}{\eta} [1- x_1 \frac{(x_2^2 - x_1 ^2)(r^2 - x_2^2)}{\eta} \frac{1} {{\lambda_1 - \lambda_2 }}\frac{f_1 }{\lambda_1}] \notag \\
&+ 2x_1 x_2 \frac{x_2^2 - x_1 ^2}{\eta} \frac{1} {{\lambda_1 - \lambda_2 }} \frac{f}{\lambda_1}[ \beta \frac{{\eta _1 }}{{\eta }} + \frac{{g'}}{g}u_1 u_{11}]\notag \\
=&\frac{- 2x_2(r^2  - x_1^2)}{\eta} [1- x_1 \frac{(x_2^2 - x_1 ^2)(r^2 - x_2^2)}{\eta} \frac{1} {{\lambda_1 - \lambda_2 }}(\frac{f_1 }{\lambda_1}+  \frac{{g'}}{g}u_1 f)] \notag \\
&+ 2x_1 x_2 \frac{x_2^2 - x_1 ^2}{\eta} \frac{1} {{\lambda_1 - \lambda_2 }} \frac{f}{\lambda_1} \beta [\frac{- 2x_1(r^2  - x_2^2)}{\eta} + 2x_1 x_2 \frac{x_2^2 - x_1 ^2}{\eta} \frac{u_{112}} {{\lambda_1 - \lambda_2 }}]\notag \\
=&\frac{- 2x_2(r^2  - x_1^2)}{\eta} [1- x_1 \frac{(x_2^2 - x_1 ^2)(r^2 - x_2^2)}{\eta} \frac{1} {{\lambda_1 - \lambda_2 }}(\frac{f_1 }{\lambda_1}+  \frac{{g'}}{g}u_1 f - \frac{f}{\lambda_1} \beta \frac{2x_1(r^2  - x_2^2)}{\eta})] \notag \\
&+[ 2x_1 x_2 \frac{x_2^2 - x_1 ^2}{\eta} ]^2 \frac{f} {{(\lambda_1 - \lambda_2 )^2}} \beta [ - \beta \frac{ \eta_{2}}{ \eta} - \frac{{g'}}{g}u_2 u_{22}], \notag
\end{align}
then we can get
\begin{align}\label{3.17}
&[1+ \beta ^2 ( 2x_1 x_2 \frac{x_2^2 - x_1 ^2}{\eta} )^2 \frac{f} {{(\lambda_1 - \lambda_2 )^2}} ]\frac{{\eta _2 }}{{\eta }}  \notag \\
=&\frac{- 2x_2(r^2  - x_1^2)}{\eta} [1- x_1 \frac{(x_2^2 - x_1 ^2)(r^2 - x_2^2)}{\eta} \frac{1} {{\lambda_1 - \lambda_2 }}(\frac{f_1 }{\lambda_1}+  \frac{{g'}}{g}u_1 f - \frac{f}{\lambda_1} \beta \frac{2x_1(r^2  - x_2^2)}{\eta})] \notag \\
&+\frac{- 2x_2(r^2  - x_1^2)}{\eta} \frac{2x_1^2x_2(x_2^2 - x_1 ^2)^2 (r^2  - x_2^2)}{\eta^2}  \frac{f} {{(\lambda_1 - \lambda_2 )^2}} \beta  \frac{{g'}}{g}u_2\frac{f} {{\lambda_1}}.
\end{align}
Hence
\begin{align}
\beta f \frac{g'}{g}  \frac{{\eta_{2} }}{{\eta }} u_2
\geq& - \beta f \frac{g'}{g} \frac{ 2 r^3}{\eta}[ 1+ \frac{ 2 r^5}{\eta \lambda_1} (\frac{|f_1 |}{\lambda_1}+  \frac{{g'}}{g} |u_1| f + \frac{2 \beta f r^3}{\eta\lambda_1} ) +\frac{16 \beta r^9}{\eta^2}  \frac{f^2} {{\lambda_1^3}}   \frac{{g'}}{g} |u_2| ] |u_2| \notag \\
\geq& - \beta f \frac{g'}{g} \frac{ 2 r^3}{\eta}[ 1+ \frac{1}{4} +  \frac{1}{4} + \frac{1}{4} + \frac{1}{4} ] |u_2|\notag \\
=& - 4\beta f \frac{g'}{g} \frac{ r^4}{\eta} \frac{|u_2|}{r}.
\end{align}
In fact, $ \frac{{\eta _2 }}{{\eta  }} \approx \frac{- 2x_2(r^2  - x_1^2)}{\eta}$ if $\eta \lambda_1$ is big enough. And we can get from \eqref{3.17}
\begin{align}
 \frac{{\eta _2 ^2 }}{{\eta ^2 }} \geq&  \left\{[1+ \beta ^2 ( 2x_1 x_2 \frac{x_2^2 - x_1 ^2}{\eta} )^2 \frac{f} {{(\lambda_1 - \lambda_2 )^2}} ]\frac{{\eta _2 }}{{\eta }} \right\}^2 [1- \beta ^2 ( 2x_1 x_2 \frac{x_2^2 - x_1 ^2}{\eta} )^2 \frac{f} {{(\lambda_1 - \lambda_2 )^2}} ]^2 \notag \\
 \geq& \left[\frac{- 2x_2(r^2  - x_1^2)}{\eta} \right]^2 [ 1- \frac{ 2r^5}{\eta \lambda_1} (\frac{|f_1 |}{\lambda_1}+  \frac{{g'}}{g} |u_1| f + \frac{2 \beta f r^3}{\eta\lambda_1} )-\frac{8\beta r^9}{\eta^2 \lambda_1^2}\frac{{g'}}{g}|u_2|\frac{f^2} {{\lambda_1}} ]^2[1- \beta ^2  \frac{16 r^8}{\eta^2}  \frac{f} {{\lambda_1^2}} ]^2 \notag \\
 \geq& \left[\frac{- 2x_2(r^2  - x_1^2)}{\eta} \right]^2 [ 1- \frac{1}{10^3} - \frac{64}{10^3}- \frac{1}{10}- \frac{1}{10^3}]^2[1- \frac{1}{10^3}]^2 \notag \\
 \geq& \frac{1}{2}\left[\frac{- 2x_2(r^2  - x_1^2)}{\eta} \right]^2. \notag
\end{align}

Taking second derivatives of $\eta$, we can get
\begin{align*}
\eta _{ii}  =& [ - 2 + 2\left\langle {x,\tau } \right\rangle \left\langle {x,\tau } \right\rangle _{ii}  + 2\left\langle {x,\tau } \right\rangle _i \left\langle {x,\tau } \right\rangle _i ][r^2  - \left\langle {x,\tau } \right\rangle ^2 ]  \\
&+ 2[ - 2x_i  + 2\left\langle {x,\tau } \right\rangle \left\langle {x,\tau } \right\rangle _i ][ - 2\left\langle {x,\tau } \right\rangle \left\langle {x,\tau } \right\rangle _i ]  \\
&+ [r^2  - |x|^2  + \left\langle {x,\tau } \right\rangle ^2 ][ - 2\left\langle {x,\tau } \right\rangle \left\langle {x,\tau } \right\rangle _{ii}  - 2\left\langle {x,\tau } \right\rangle _i \left\langle {x,\tau } \right\rangle _i ]  \\
=& [ - 2 + 2x_1 \left\langle {x,\tau } \right\rangle _{ii}  + 2(\delta_{i1} +  \left\langle {x, \partial_i \tau } \right\rangle)^2 ][r^2  - x_1 ^2 ]  \\
&+ 2[ - 2x_i  + 2x_1 (\delta_{i1} +  \left\langle {x, \partial_i \tau } \right\rangle) ][ - 2x_1 (\delta_{i1} +  \left\langle {x, \partial_i \tau } \right\rangle) ] \\
&+ (r^2 -x_2^2)[ - 2x_1 \left\langle {x,\tau } \right\rangle _{ii}  - 2(\delta_{i1} +  \left\langle {x, \partial_i \tau } \right\rangle)^2 ],
\end{align*}
so
\begin{align}
\label{3.19}\eta _{11}  =& - 2(r^2-x_2^2) - 2x_1 (x_1^2 -x_2^2 )\left\langle {x,\tau } \right\rangle _{11} \notag \\
&+ ( 4x_2^2 -12 x_1^2 )\left\langle {x, \partial_1 \tau } \right\rangle + ( 2x_2^2 -10 x_1^2 )  \left\langle {x, \partial_1 \tau } \right\rangle^2,  \\
\label{3.20}\eta _{22}  =& - 2(r^2  - x_1 ^2 ) - 2x_1 (x_1^2 -x_2^2 )\left\langle {x,\tau } \right\rangle _{22} \notag \\
&+ 8 x_1 x_2  \left\langle {x, \partial_2 \tau } \right\rangle + ( 2x_2^2 -10 x_1^2 ) \left\langle {x, \partial_2 \tau } \right\rangle^2.
\end{align}
Hence
\begin{align}\label{3.21}
\beta [\frac{f}{\lambda_1}\frac{{\eta _{11}}}{\eta } + \lambda_1 \frac{{\eta _{22}
}}{\eta } ] =& -2\beta [\frac{f}{\lambda_1}\frac{{r^2-x_2^2}}{\eta } + \lambda_1 \frac{{r^2-x_1^2}}{\eta } ]  \notag\\
&-2\beta \frac{{x_1 (x_1^2 -x_2^2 )}}{\eta } [\frac{f}{\lambda_1}\left\langle {x,\tau } \right\rangle _{11} + \lambda_1 \left\langle {x,\tau } \right\rangle _{22}]  \notag \\
&+\beta\frac{f}{\lambda_1} [ \frac{{x_2 ( 4x_2^2 -12 x_1^2 ) }}{\eta }\frac{u_{112}} {{\lambda_1 - \lambda_2 }}+\frac{{x_2^2 (2x_2^2 -10 x_1^2) }}{\eta }(\frac{u_{112}} {{\lambda_1 - \lambda_2 }})^2 ] \\
& +\beta \lambda_1 [ \frac{{ 8x_1 x_2^2}}{\eta } \frac{u_{221}} {{\lambda_1 - \lambda_2 }} + \frac{{x_2^2 ( 2x_2^2 -10 x_1^2 ) }}{\eta }(\frac{u_{221}} {{\lambda_1 - \lambda_2 }})^2  ]. \notag
\end{align}
Direct calculations yield
\begin{align*}
 \left\langle {x,\tau } \right\rangle _{11}  =& \frac{{\partial ^2 }}
{{\partial x_1 ^2 }}[\sum\limits_{m = 1}^2 {x_m \tau _m } ] = 2\frac{{\partial \tau _1 }}
{{\partial x_1 }} + \sum\limits_{m = 1}^2 {x_m \frac{{\partial ^2 \tau _m }}
{{\partial x_1 ^2 }}}  \\
=& 2\frac{{\partial \tau _1 }}
{{\partial u_{pq} }}u_{pq1}  + \sum\limits_{m = 1}^2 {x_m [\frac{{\partial \tau _m }}
{{\partial u_{pq} }}u_{pq11}  + \frac{{\partial ^2 \tau _m }}
{{\partial u_{pq} \partial u_{rs} }}u_{pq1} u_{rs1} ]}  \\
=& 0+x_1 \frac{{\partial ^2 \tau _1 }}{{\partial u_{pq} \partial u_{rs} }}u_{pq1} u_{rs1} +x_2 [\frac{{\partial \tau _2 }}
{{\partial u_{pq} }}u_{pq11}  + \frac{{\partial ^2 \tau _2 }}
{{\partial u_{pq} \partial u_{rs} }}u_{pq1} u_{rs1} ] \\
=&  - x_1\left[ {\frac{{u_{112} }} {{\lambda _1  - \lambda _2 }}} \right] ^2 + x_2 \left[\frac{{ 1}} {{\lambda _1  - \lambda _2 }} \right] u_{1211}+ 2x_2 \left[-\frac{{ u_{112} u_{111} }} {{(\lambda _1  - \lambda _2)^2 }} +\frac{{ u_{112} u_{221} }} {{(\lambda _1  - \lambda _2)^2 }} \right],
\end{align*}
Similarly, we have
\begin{align*}
\left\langle {x,\tau } \right\rangle _{22}  =& \frac{{\partial ^2 }}
{{\partial x_2 ^2 }}[\sum\limits_{m = 1}^2 {x_m \tau _m } ] = 2\frac{{\partial \tau _2 }}
{{\partial x_2 }} + \sum\limits_{m = 1}^2 {x_m \frac{{\partial ^2 \tau _m }}
{{\partial x_1 ^2 }}}  \\
=& 2\frac{{\partial \tau _2 }}
{{\partial u_{pq} }}u_{pq2}  + \sum\limits_{m = 1}^2 {x_m [\frac{{\partial \tau _m }}
{{\partial u_{pq} }}u_{pq22}  + \frac{{\partial ^2 \tau _m }}
{{\partial u_{pq} \partial u_{rs} }}u_{pq2} u_{rs2} ]}  \\
=& 2\left[  \frac{{1 }}
{{\lambda _1  - \lambda _2}} \right]u_{221}  - x_1\left[ {\frac{{u_{221} }} {{\lambda _1  - \lambda _2 }}} \right] ^2 \\
 &+ x_2 \left[\frac{{ 1}} {{\lambda _1  - \lambda _2 }} \right] u_{1222}+ 2x_2 \left[-\frac{{ u_{112} u_{221} }} {{(\lambda _1  - \lambda _2)^2 }} +\frac{{ u_{222} u_{221} }} {{(\lambda _1  - \lambda _2)^2 }} \right],
\end{align*}
then
\begin{align}\label{3.22}
 \frac{f}{\lambda_1}\left\langle {x,\tau } \right\rangle _{11} + \lambda_1 \left\langle {x,\tau } \right\rangle _{22} =&  - x_1\frac{f}{\lambda_1}\left[ {\frac{{u_{112} }} {{\lambda _1  - \lambda _2 }}} \right] ^2 + 2\lambda_1\left[  \frac{{ u_{221}}}
{{\lambda _1  - \lambda _2}} \right] - x_1 \lambda_1\left[ {\frac{{u_{221} }} {{\lambda _1  - \lambda _2 }}} \right] ^2 \notag \\
&+ x_2 \left[\frac{{ 1}} {{\lambda _1  - \lambda _2 }} \right] [f_{12} + f_1 \frac{ u_{112}}{u_{11}} - f_2 ( \frac{u_{111}}{u_{11}})] \notag \\
&+ 2x_2 \left[-\frac{{ u_{112} }} {{(\lambda _1  - \lambda _2)^2 }}f_{1} +\frac{{ u_{221} }} {{(\lambda _1  - \lambda _2)^2 }}f_{2} \right].
\end{align}

From \eqref{3.21} and \eqref{3.22}, we can get
\begin{align*}
&\beta [\frac{f}{\lambda_1}\frac{{\eta _{11}}}{\eta } + \lambda_1 \frac{{\eta _{22}
}}{\eta } ] \\
=& -2\beta [\frac{f}{\lambda_1}\frac{{r^2-x_2^2}}{\eta } + \lambda_1 \frac{{r^2-x_1^2}}{\eta } ] -2\beta \frac{{x_1 x_2(x_1^2 -x_2^2 )}}{\eta } \left[\frac{{ 1}} {{\lambda _1  - \lambda _2 }} \right] [f_{12} - f_2 ( \frac{u_{111}}{u_{11}})] \notag \\
&+(\frac{u_{112}} {{\lambda_1 - \lambda_2 }})^2  [2\beta \frac{f}{\lambda_1} \frac{{x_1^2 (x_1^2 -x_2^2 )}}{\eta }+\beta\frac{f}{\lambda_1}\frac{{x_2^2 ( 2x_2^2 -10 x_1^2 ) }}{\eta }] \\
&+\frac{u_{112}} {{\lambda_1 - \lambda_2 }} [-2\beta\frac{{x_1 (x_1^2 -x_2^2 )}}{\eta }(x_2\frac{f_1}{\lambda_1}-2x_2\frac{f_1}{\lambda_1 - \lambda_2 })+\beta\frac{f}{\lambda_1}\frac{{x_2 ( 4x_2^2 -12 x_1^2 ) }}{\eta }]  \\
& +(\frac{u_{221}} {{\lambda_1 - \lambda_2 }})^2 [2\beta \lambda_1\frac{{x_1^2 (x_1^2 -x_2^2 )}}{\eta }+\beta \lambda_1\frac{{ x_2^2 (2x_2^2 -10 x_1^2 ) }}{\eta }] \\
&+\frac{u_{221}} {{\lambda_1 - \lambda_2 }}  [-2\beta \frac{{x_1 (x_1^2 -x_2^2 )}}{\eta }(2 \lambda_1 + 2x_2\frac{f_2}{\lambda_1 - \lambda_2 })+\beta \lambda_1 \frac{{8x_1 x_2^2 }}{\eta } ]  \\
\geq& -2\beta \lambda_1 \frac{{r^2-x_1^2}}{\eta } - 2\beta f \frac{{r^2}}{\eta \lambda_1} - 2\beta |f_{12}| \frac{{r^4}}{\eta (\lambda_1-\lambda_2)}- 2\beta |f_{2}| \frac{{r^4}}{\eta (\lambda_1-\lambda_2)} | \frac{u_{111}}{u_{11}}|\\
& -(\frac{u_{112}} {{(\lambda_1-\lambda_2) }})^2  \beta \frac{f}{\lambda_1} \frac{{8r^4}}{\eta } -|\frac{u_{112}} {{\lambda_1-\lambda_2}} | [6 \beta\frac{|f_1|}{\lambda_1-\lambda_2} \frac{{r^4}}{\eta } +\beta\frac{f}{\lambda_1}\frac{{12r^3}}{\eta } ]  \\
& -\frac{1} {{(\lambda_1-\lambda_2)^2}} [u_{221}]^2 \beta \lambda_1\frac{{8 r^4 }}{\eta }-\frac{1} {{\lambda_1-\lambda_2 }}  |u_{221}|[4 \beta \frac{{r^4}}{\eta }\frac{|f_2|}{\lambda_1}+\beta \lambda_1 \frac{{12 r^3  }}{\eta } ]\\
\geq& -2\beta \lambda_1 \frac{{r^2-x_1^2}}{\eta } - 2\beta f \frac{{r^2}}{\eta \lambda_1} - 4\beta |f_{12}| \frac{{r^4}}{\eta \lambda_1}- 4\beta |f_{2}| \frac{{r^4}}{\eta \lambda_1} | \frac{u_{111}}{u_{11}}|\\
& -(\frac{u_{112}} {{u_{11} }})^2  \beta \frac{f}{\lambda_1} \frac{{16r^4}}{\eta } -|\frac{u_{112}} {{u_{11}}} | [12 \beta\frac{|f_1|}{\lambda_1} \frac{{r^4}}{\eta } +\beta\frac{f}{\lambda_1}\frac{{24r^3}}{\eta } ]  \\
& -2[\frac{f_1 ^2}{\lambda_1^4} + \frac{f^2}{\lambda_1^4} (\frac{u_{111}}{u_{11}})^2] \beta \lambda_1\frac{{16 r^4 }}{\eta }- [\frac{|f_1| }{\lambda_1^2} + \frac{f}{\lambda_1^2} |\frac{u_{111}}{u_{11}}|][8 \beta \frac{{r^4}}{\eta }\frac{|f_2|}{\lambda_1}+\beta \lambda_1 \frac{{24 r^3  }}{\eta } ]\\
\geq& -2\beta \lambda_1 \frac{{r^2-x_1^2}}{\eta } - 2\beta f \frac{{r^2}}{\eta \lambda_1} - 4\beta |f_{12}| \frac{{r^4}}{\eta \lambda_1} -32 \beta \frac{f_1 ^2 r^4}{ \eta \lambda_1^3} -8 \beta \frac{|f_1 f_2| r^4}{ \eta \lambda_1^3} -24 \beta \frac{ |f_1|r^3}{ \eta \lambda_1}\\
&-\frac{\lambda_1}{2}(\frac{u_{112}} {{u_{11} }})^2 [ \frac{32 \beta f r^4 }{\eta \lambda_1^2} +\frac{12 \beta r^4}{\eta \lambda_1 ^2} +\frac{24 \beta f r^4}{\eta \lambda_1^2} ]- \frac{\lambda_1}{2}[\frac{12 \beta |f_1|^2r^4}{\eta \lambda_1 ^2} +\frac{24 \beta f r^2}{\eta \lambda_1^2} ]  \\
& - \frac{f}{2\lambda_1}(\frac{u_{111}}{u_{11}})^2[ \frac{64 \beta f r^4 }{\eta \lambda_1^2} +\frac{{16\beta r^4 }}{\eta \lambda_1^2}+ \frac{1}{4}+ \frac{1}{4}] - \frac{f}{2\lambda_1} [ \frac{{16\beta r^4 |f_2|^2}}{\eta \lambda_1^2}+ (\frac{{48 \beta r^3  }}{\eta })^2+ (\frac{{8 \beta r^4  }}{\eta }\frac{|f_{2}|}{f})^2 ]\\
\geq& -2\beta \lambda_1 \frac{{r^2-x_1^2}}{\eta } - 2\beta f \frac{{r^2}}{\eta \lambda_1} - 4\beta |f_{12}| \frac{{r^4}}{\eta \lambda_1} -32 \beta \frac{f_1 ^2 r^4}{ \eta \lambda_1^3} -8\beta \frac{|f_1 f_2| r^4}{ \eta \lambda_1^3} -24 \beta \frac{ |f_1|r^3}{ \eta \lambda_1}\\
&-\frac{\lambda_1}{4}(\frac{u_{112}} {{u_{11} }})^2 - [\frac{6 \beta |f_1|^2r^4}{\eta \lambda_1} +\frac{12 \beta f r^2}{\eta \lambda_1} ]  \\
& - \frac{f}{2\lambda_1}(\frac{u_{111}}{u_{11}})^2 -  [ \frac{{8\beta f r^4 |f_2|^2}}{\eta \lambda_1^3}+ \frac{{(48 \beta)^2 f r^6  }}{\eta^2\lambda_1 }+ \frac{{32 \beta^2 |f_{2}|^2 r^8  }}{\eta^2\lambda_1 f}].
\end{align*}

Now we just need to estimate $-2\beta \lambda_1 \frac{{r^2-x_1^2}}{\eta }$. If $x_2^2 \leq \frac{r^2}{2}$, we can get
\begin{align*}
-2\beta \lambda_1 \frac{{r^2-x_1^2}}{\eta }= - \frac{{8}}{r^2-x_2^2 } \lambda_1 \geq - \frac{{16}}{r^2} \lambda_1 \geq - \frac{1}{2}\frac{{c_0}}{r^2} f\lambda_1 = -  \frac{1}{2}\frac{{g'}}{g} f\lambda_1.
\end{align*}
If $x_2^2 \geq \frac{r^2}{2}$, we can get
\begin{align*}
-2\beta \lambda_1 \frac{{r^2-x_1^2}}{\eta }=& - \frac{{8}}{r^2-x_2^2 } \lambda_1 \geq -\frac{{x_2^2}}{r^2-x_2^2 } \frac{{8}}{r^2-x_2^2 } \lambda_1 = -\beta \lambda_1 \frac{1}{2}[\frac{{2x_2}}{r^2-x_2^2 }]^2  \\
\geq& -\beta \lambda_1 [\frac{\eta_2}{\eta}]^2 .
\end{align*}
Hence
\begin{align}\label{3.23}
-2\beta \lambda_1 \frac{{r^2-x_1^2}}{\eta }\geq  -  \frac{1}{2}\frac{{g'}}{g} f\lambda_1-\beta \lambda_1 [\frac{\eta_2}{\eta}]^2 .
\end{align}
and
\begin{align}\label{3.24}
\beta [\frac{f}{\lambda_1}\frac{{\eta _{11}}}{\eta } + \lambda_1 \frac{{\eta _{22}
}}{\eta } ]
\geq& -  \frac{1}{2}\frac{{g'}}{g} f\lambda_1-\beta \lambda_1 [\frac{\eta_2}{\eta}]^2 - \frac{f}{2\lambda_1}(\frac{u_{111}}{u_{11}})^2 -\frac{\lambda_1}{4}(\frac{u_{112}} {{u_{11} }})^2  \notag\\
&- 2\beta f \frac{{r^2}}{\eta \lambda_1} - 4\beta |f_{12}| \frac{{r^4}}{\eta \lambda_1} -32 \beta \frac{f_1 ^2 r^4}{ \eta \lambda_1^3} -8\beta \frac{|f_1 f_2| r^4}{ \eta \lambda_1^3} -24 \beta \frac{ |f_1|r^3}{ \eta \lambda_1} \notag \\
&- [\frac{6 \beta |f_1|^2r^4}{\eta \lambda_1} +\frac{12 \beta f r^2}{\eta \lambda_1} ]  -  [ \frac{{8\beta f r^4 |f_2|^2}}{\eta \lambda_1^3}+ \frac{{(48 \beta)^2 f r^6  }}{\eta^2\lambda_1 }+ \frac{{32 \beta^2 |f_{2}|^2 r^8  }}{\eta^2\lambda_1 f}].
\end{align}

\end{proof}

Now we continue to prove Theorem \ref{th1.1}.  From \eqref{3.10} and Lemma \ref{lem3.1}, we can get
\begin{align}
0 \ge& \sum\limits_{i = 1}^2 {F^{ii} \varphi _{ii} }  \notag\\
\geq&  \frac{1}{2}\frac{{g'}}{g} f \lambda_1  - 2\beta f \frac{{r^2}}{\eta \lambda_1} - 4\beta |f_{12}| \frac{{r^4}}{\eta \lambda_1} -32 \beta \frac{f_1 ^2 r^4}{ \eta \lambda_1^3} -8\beta \frac{|f_1 f_2| r^4}{ \eta \lambda_1^3} -24 \beta \frac{ |f_1|r^3}{ \eta \lambda_1} \notag \\
&- [\frac{6 \beta |f_1|^2r^4}{\eta \lambda_1} +\frac{12 \beta f r^2}{\eta \lambda_1} ]  -  [ \frac{{8\beta f r^4 |f_2|^2}}{\eta \lambda_1^3}+ \frac{{(48 \beta)^2 f r^6  }}{\eta^2\lambda_1 }+ \frac{{32 \beta^2 |f_{2}|^2 r^8  }}{\eta^2\lambda_1 f}]\notag \\
& - 4\beta f \frac{g'}{g} \frac{ r^4}{\eta} \frac{|u_2|}{r} - \frac{{g'}}{g} |\nabla u| |\nabla f| - \frac{| f_{11} |}{{\lambda_1 }}- 2\frac{f_1^2}{f\lambda_1} -\frac{8\beta f r^6 }{\eta^2\lambda_1}\notag\\
\geq&  \frac{c_0 m}{2r^2} \lambda_1   - \frac{{8 r^2\cdot M}}{\eta \lambda_1} -  \frac{{32r^2 \cdot r^2|f_{12}|}}{\eta \lambda_1} - \frac{128r^2 \cdot r^2 f_1 ^2 }{ \eta \lambda_1^3} - \frac{32r^2 \cdot r^2|f_1 f_2|}{ \eta \lambda_1^3} - \frac{ 96r^2 \cdot r|f_1|}{ \eta \lambda_1} \notag \\
&- \frac{24r^2 \cdot r^2|f_1|^2}{\eta \lambda_1} -\frac{48r^2 \cdot M}{\eta \lambda_1}   -   \frac{{32r^2 \cdot M \cdot r^2 |f_2|^2}}{\eta \lambda_1^3}- \frac{{192^2\cdot M \cdot r^6  }}{\eta^2\lambda_1 }- \frac{{512 r^6 \cdot r^2 |f_{2}|^2 \cdot\frac{1}{ m}}}{\eta^2\lambda_1} \notag \\
& -  \frac{ 16r^2 \cdot c_0 m}{\eta} \frac{|u_2|}{r} - \frac{{c_0}}{r^2} \cdot r |\nabla f| \cdot \frac{|\nabla u|}{r} - \frac{| f_{11} |}{{\lambda_1 }}- 2\frac{f_1^2}{m\lambda_1} - \frac{32 M r^6 }{\eta^2\lambda_1}. \notag
\end{align}
So we can get
\begin{align} \label{3.25}
\eta \lambda_1 \leq&   C (1 + \frac{|\nabla u| }{r}) r^4.
\end{align}
where $C$ is a positive constant depending only on $c_0, m, M, r|\nabla f|$, and $r^2 |\nabla^2 f|$.
So we can easily get
\begin{align*}
u_{\tau(0) \tau(0)}(0) \leq& \frac{1}{r^{4 \beta}} \phi (0)\leq \frac{1}{r^{4\beta}} \phi (x_0)  \leq  C (1 + \frac{\sup |D u| }{r}) e^{c_0\frac{\sup |Du|^2}{r^2}} \\
\leq& C e^{(c_0+2)\frac{\sup |Du|^2}{r^2}},
\end{align*}
and
\begin{align} \label{3.26}
|u_{\xi \xi} (0) |\leq u_{\tau(0) \tau(0)}(0) \leq  C e^{(c_0+2)\frac{\sup |Du|^2}{r^2}}, \quad \forall ~\xi \in \mathbb{S}^{1}.
\end{align}
Then we prove Theorem \ref{th1.1} under the condition $\eta \lambda_1 \geq 10^3 (1+ M + r \sup |\nabla f| + \frac{M}{m}\frac{\sup |Du|}{r})r^4$.
Hence Theorem \ref{th1.1} holds.

\begin{remark}
The eigenvector field $\tau$ is important. In fact, it is well-defined when the largest eigenvalue is distinct with others, and $\tau$ depends only on the adjoint matrix. For the Monge-Amp\`{e}re equation in dimension $n \geq 3$, we do not know whether the largest eigenvalue is distinct with others. So our method is not suitable for this case.
\end{remark}

\textbf{Acknowledgement}.
The authors would like to express sincere gratitude to Prof. Xi-Nan Ma for the constant encouragement in this subject.

\end{document}